\documentclass{amsart}
\title{On a conjecture of Khoroshkin and Tolstoy}

\author[A. Appel]{Andrea Appel}
\address{
Dipartimento di Scienze Matematiche, Fisiche e Informatiche,	
Universit\`{a} degli Studi di Parma,
43124 Parma, Italy}
\email{andrea.appel@unipr.it}

\author[S. Gautam]{Sachin Gautam}
\address{Department of Mathematics, The Ohio State University,
Columbus, OH 43210, USA}
\email{gautam.42@osu.edu}

\author[C. Wendlandt]{Curtis Wendlandt}
\address{Department of Mathematics and Statistics,
University of Saskatchewan,
Saskatoon, SK S7N 5E6,
Canada}
\email{wendlandt@math.usask.ca}

\subjclass[2020]{ 
	Primary: 81R10. 
	Secondary: 17B37. 
}

\usepackage{amsmath,amssymb,mathtools}
\usepackage[usenames,dvipsnames]{color}
\usepackage{enumitem}
\usepackage{comment}
\usepackage{hyperref}
\hypersetup{
    colorlinks=true,       			
    linkcolor=blue, 
    citecolor=blue,        			
    filecolor=blue,      				
    urlcolor=blue           			
}

\usepackage{float}
\usepackage[all]{xy}
\usepackage{tikz-cd} 
\xyoption{arc}
\newtheorem*{thm}{Theorem}

\theoremstyle{definition}

\newtheorem*{rem}{Remark}

\numberwithin{equation}{section}

\numberwithin{figure}{section}

\newcommand{\Pseries}[2]{#1[\negthinspace[#2]\negthinspace]}

\newcommand{\Top}{\operatorname{T}}

\newcommand{\sfr}{\mathsf{r}^-}

\newcommand{\id}{\mathbf{1}}

\newcommand{\lp}{\left(}
\newcommand{\rp}{\right)}

\newcommand{\g}{\mathfrak{g}}
\newcommand{\h}{\mathfrak{h}}

\newcommand{\Sym}{\mathfrak{S}}

\newcommand{\vup}[1]{v_+}
\newcommand{\vdown}[1]{v_-}

\newcommand{\bfA}{\mathbf{A}}

\newcommand{\RR}{\mathcal{R}}

\newcommand{\C}{\mathbb{C}}
\newcommand{\nC}{\mathbb{C}^{\times}}
\newcommand{\N}{\mathbb{Z}_{\geq 0}}

\newcommand{\Z}{\mathbb{Z}}

\newcommand {\wh}[1]{\widehat{#1}}

\newcommand {\ul}[1]{\underline{#1}}

\newcommand{\ad}{\operatorname{ad}}

\newcommand {\aand}{\qquad\text{and}\qquad}

\newcommand{\op}{\operatorname{op}}

\newcommand {\fd}{finite-dimensional }

\newcommand {\Comment}[1]{}
\newcommand {\Omit}[1]{}

\newcommand {\Yhg}{Y_\hbar(\g)}
\newcommand{\Ynaughthg}{Y^0_{\hbar}(\g)}
\newcommand {\Ypmhg}{Y^\pm_\hbar(\g)}

\renewcommand {\sl}{\mathfrak{sl}}

\newcommand {\eg}{{\it e.g., }}

\newcommand {\ie}{{\it i.e.}, }
\newcommand{\bfI}{{\mathbf I}}

\DeclareMathAlphabet\EuScript{U}{eus}{m}{n}

\SetMathAlphabet\EuScript{bold}{U}{eus}{b}{n}

\newcommand{\boldR}{{\boldsymbol{\EuScript{R}}}}
\newcommand{\re}{\scriptscriptstyle{\text{re}}}

\newcommand{\im}{\scriptscriptstyle{\text{im}}}

\allowdisplaybreaks

\begin{document}

\begin{abstract}
We prove a \emph{no-go theorem} on the factorization 
of the lower triangular part in the Gaussian
decomposition of the Yangian's universal $R$--matrix,
yielding a negative answer to a conjecture of Khoroshkin and Tolstoy 
from [{\em Lett.~Math. Phys.} {\bf 36} 1996]. 
\end{abstract}

\maketitle

\section{Introduction}\label{sec:Intro}
Let $\g$ be a \fd complex semisimple Lie algebra
and $\Yhg$ the associated Yangian as defined by Drinfeld in
\cite{drinfeld-yangian-qaffine}.
Let $\RR(s)\in\Pseries{\lp\Yhg\otimes\Yhg\rp}{s^{-1}}$ be Drinfeld's
universal $R$--matrix \cite{drinfeld-qybe}. 
A constructive proof of the existence of $\RR(s)$ was recently obtained
by the last two authors and V. Toledano Laredo in \cite[Thm.~7.4]{GTLW}.
This was achieved by providing a direct construction of the 
components of the Gaussian decomposition 
\begin{equation*}
\RR(s) = \RR^+(s)\RR^0(s)\RR^-(s).
\end{equation*}
 The diagonal part $\RR^0(s)$ was first defined by the second author and
 Toledano Laredo in \cite[Thm.~5.9]{sachin-valerio-III} as a meromorphic function of $s$ acting on the tensor product of any two \fd representations of $\Yhg$.
Its explicit expansion as a formal series in $s^{-1}$ with coefficients in $\Yhg\otimes \Yhg$ was later provided in \cite[Thm.~6.7]{GTLW}. 
The component $\RR^-(s)$ is then obtained in \cite[Thm.~4.1]{GTLW} as the 
unique (unipotent, zero--weight) solution to a system of linear equations 
(see also \S \ref{ssec:IE} below).
Finally, the upper/lower triangular parts
are related by $\RR^+(s) = \RR^-_{21}(-s)^{-1}$.

In this paper, we focus our attention on $\RR^-(s)$.
When $\g=\sl_2$, a closed form formula for $\RR^-(s)$ is given in \cite[Thm.~5.5]{GTLW} (see also \cite[Lemma~5.1]{khoroshkin-tolstoy}), 
but no explicit formula is known in higher rank.
Following a conjecture of Khoroshkin and Tolstoy 
\cite[Conjecture p.~393]{khoroshkin-tolstoy}, it
was expected that $\RR^-(s)$ could be
expressed as an ordered product over the set of positive roots of $\g$ of $\sl_2$-type components (cf.~\cite[(5.3)]{khoroshkin-tolstoy}, \cite[\S5]{Stukopin07} and in particular \cite[Thm.~6.2]{soloviev}). 
The main result of this paper shows that
$\RR^-(s)$ does not admit any such factorization
for $\g\not\simeq\mathfrak{sl}_2$ (Theorem~\ref{thm:impossible}).

The conjecture of Khoroshkin and Tolstoy mentioned 
above refers to the lower triangular part of the universal $R$--matrix of the
{Yangian double} (cf.~\S \ref{ssec:KTC}). In \cite[\S2]{khoroshkin-tolstoy} 
the latter was conjectured to be isomorphic to the restricted quantum 
double of $\Yhg$. This conjecture was recently proven by the third author in \cite[Thm.~8.4]{Curtis-DYg} and led to the identification of the 
underlying $R$-matrices (and their components) in \cite[Thm.~9.6]{Curtis-DYg} (see also \cite{Stukopin07}).
Therefore, relying on these results,
the conjecture of Khoroshkin and Tolstoy is directly disproved by our main
result.

\noindent {\bf Outline of the paper.} In Section \ref{sec:YR} we give a 
brief overview of the Yangian $\Yhg$ and its universal $R$--matrix, with emphasis on the factor $\RR^-(s)$. 
In Section \ref{sec:impossible}, we state and prove our main result (Theorem~\ref{thm:impossible}).
In Section \ref{sec:rks}, we show that it yields a negative answer  
to the Khoroshkin--Tolstoy conjecture.
Finally, in Section \ref{ssec:qla}  we discuss analogies and differences
with the factorization problem for quantum affine algebras, referring 
in particular to the work of Damiani \cite[Thm. 2]{damiani}.

\noindent {\bf Acknowledgments.} This project was completed during
a \emph{Research in Pairs} visit at the Mathematisches Forschungsinstitut Oberwolfach, October
23-30, 2021 (2143R: {\em $R$--matrices
and Weyl group symmetries}). We are grateful to the entire
MFO team for providing us an exceptional research environment. We would
also like to thank Sasha Tsymbaliuk for 
various discussions about the factorization of universal
$R$--matrices in the trigonometric setting. AA was supported in part by the Programme \emph{FIL} of the University of Parma 
co-sponsored by Fondazione Cariparma. SG was supported through the Simons foundation
collaboration grant 526947. CW gratefully acknowledges the support of the Natural Sciences 
and Engineering Research Council of Canada (NSERC), provided via the Discovery Grants Program 
(Grant RGPIN-2022-03298 and DGECR-2022-00440).

\section{Background on Yangians and $R$--matrices}\label{sec:YR}

\subsection{}

Let $\g$ be a \fd complex semisimple Lie algebra and $(\cdot,\cdot)$ an
invariant symmetric non--degenerate bilinear form on $\g$. Let $\h
\subset\g$ be a Cartan subalgebra of $\g$, $\{\alpha_i\}_{i\in\bfI}\subset
\h^*$ a basis of simple roots of $\g$ relative to $\h$ and $a_{ij}=2(\alpha
_i,\alpha_j)/(\alpha_i,\alpha_i)$ the entries of the corresponding Cartan
matrix $\bfA$. Let $\Phi_+\subset \h^*$ be the corresponding set of
positive roots, and $\mathsf{Q}=\Z\Phi_+=\bigoplus_{i\in \bfI}\Z\alpha_i\subset
\h^*$ the root lattice. The semigroup $\Z_{\geq 0}\Phi_+ \subset \mathsf{Q}$
is denoted by $\mathsf{Q}_+$. We assume that $(\cdot,\cdot)$ is normalised so
that the square length of short roots is $2$. Set $d_i=(\alpha_i,\alpha_i)
/2\in\{1,2,3\}$, so that $d_ia_{ij}=d_j a_{ji}$ for any $i,j\in\bfI$.
In addition, we set $h_i=\nu^{-1}(\alpha_i)/d_i$ and choose root vectors
$x_i^\pm \in \g_{\pm \alpha_i}$ such that $[x_i^+,x_i^-]=d_ih_i$, where
$\nu:\h\to\h^*$ is the isomorphism determined by $(\cdot,\cdot)$.

\subsection{The Yangian $\Yhg$ \cite{drinfeld-yangian-qaffine}}\label{ssec: yangian}

Let $\hbar\in\nC$. The Yangian $\Yhg$ is the unital, associative $\C$--algebra
generated by elements $\{x^{\pm}_{i,r},\xi_{i,r}\}_{i\in\bfI,
r\in\N}$, subject to the following relations:
\begin{enumerate}[font=\upshape, label=(Y\arabic*)]\itemsep0.25cm
\item\label{Y1} For any $i,j\in\bfI$, $r,s\in\N$: $[\xi_{i,r}, \xi_{j,s}] = 0$.

\item\label{Y2} For $i,j\in\bfI$ and $s\in \N$: $
[\xi_{i,0}, x_{j,s}^{\pm}] = \pm d_ia_{ij} x_{j,s}^{\pm}$.

\item\label{Y3} For $i,j\in\bfI$ and $r,s\in\N$:
\[[\xi_{i,r+1}, x^{\pm}_{j,s}] - [\xi_{i,r},x^{\pm}_{j,s+1}] =
\pm\hbar\frac{d_ia_{ij}}{2}(\xi_{i,r}x^{\pm}_{j,s} + x^{\pm}_{j,s}\xi_{i,r})\ .
\]

\item\label{Y4} For $i,j\in\bfI$ and $r,s\in \N$:
\[
[x^{\pm}_{i,r+1}, x^{\pm}_{j,s}] - [x^{\pm}_{i,r},x^{\pm}_{j,s+1}]=
\pm\hbar\frac{d_ia_{ij}}{2}(x^{\pm}_{i,r}x^{\pm}_{j,s} 
+ x^{\pm}_{j,s}x^{\pm}_{i,r})\ .
\]
\item\label{Y5} For $i,j\in\bfI$ and $r,s\in \N$:
$[x^+_{i,r}, x^-_{j,s}] = \delta_{ij} \xi_{i,r+s}$.

\item\label{Y6} Let $i\not= j\in\bfI$ and set $m = 1-a_{ij}$. For any
$r_1,\cdots, r_m, s\in \N$:
\[
\sum_{\pi\in\Sym_m}
\left[x^{\pm}_{i,r_{\pi(1)}},\left[x^{\pm}_{i,r_{\pi(2)}},\left[\cdots,
\left[x^{\pm}_{i,r_{\pi(m)}},x^{\pm}_{j,s}\right]\cdots\right]\right]\right.=0.
\]
\end{enumerate}

We denote by $\Ynaughthg$ and $\Ypmhg$ the unital subalgebras of
$\Yhg$ generated by $\{\xi_{i,r}\}_{i\in\bfI, r\in\Z_{\geq 0}}$
and $\{x_{i,r}^{\pm}\}_{i\in\bfI, r\in\Z_{\geq 0}}$, respectively.
Let $Y^{\geq}_\hbar(\g)$ (resp. $Y^{\leq}_\hbar(\g)$) denote the
subalgebras of $\Yhg$ generated by $\Ynaughthg$ and $Y^+_\hbar(\g)$
(resp. $\Ynaughthg$ and $Y^-_\hbar(\g)$).

\subsection{Shift automorphism}\label{ssec: shift-yangian}

The group of translations of the complex plane acts on
$\Yhg$ by
\[\tau_a(y_r) = \sum_{s=0}^r
\left(\begin{array}{c}r\\s\end{array}\right)
a^{r-s}y_s\]
where $a\in\C$ and $y$ is one of $\xi_i,x_i^\pm$.

\subsection{}\label{ssec:ti1}

For each $i\in \bfI$, define $t_{i,1}\in \Ynaughthg$ by the formula
\begin{equation}\label{eq:tione}
t_{i,1}:=\xi_{i,1} -\frac{\hbar}{2} \xi_{i,0}^2\ .
\end{equation}
The relations \ref{Y2}--\ref{Y3} of $\Yhg$ imply that
for any $i,j\in\bfI$ and $r\in\N$, 
\begin{equation}\label{eq:h1-zero}
[t_{i,1},x_{j,r}^{\pm}] = \pm d_ia_{ij}x_{j,r+1}^{\pm}.
\end{equation}
Hence, the elements $t_{i,1}$ act as shift operators on the generators
$x_{j,r}^\pm$.

\subsection{Two embeddings $\h\to\Ynaughthg$}\label{ssec:T}

By the Poincar\'{e}--Birkhoff--Witt theorem for $\Yhg$ \cite{levendorskii-PBW} (see also 
\cite[Thm. B.6]{FiTsWe19} and
\cite[Prop. 2.2]{GRWEquiv}), there is
an embedding of $U(\g)$ into $\Yhg$, uniquely determined by
\[
x_i^\pm \mapsto x_{i,0}^\pm
\aand
d_ih_i\mapsto \xi_{i,0}\ 
\]
for each $i\in \bfI$. We shall henceforth identify $U(\g)\subset \Yhg$, with the 
above embedding implicitly understood.
Viewed as a module over $\h\subset \Yhg$, we then have
$\Yhg = \bigoplus_{\beta\in \mathsf{Q}}\Yhg_{\beta}$, where
\[
\Yhg_\beta = \{y\in\Yhg : [h,y]=\beta(h)y,\ \forall\ 
h\in\h\}\ .
\]

A second embedding $\Top:\h\to\Yhg$ is given by setting $\Top(d_ih_i)
=t_{i,1}$ for all $i\in \bfI$, where $t_{i,1}$ is defined by \eqref{eq:tione}.
Using this embedding, \eqref{eq:h1-zero} can be written as
\begin{equation}\label{eq:T}
[\Top(h),x_{i,r}^{\pm}] = \pm \alpha_i(h)x_{i,r+1}^{\pm}\quad \forall\;
h\in\h\ .
\end{equation}

\subsection{Coproduct}\label{ssec:delta}

We now recall the definition of the standard coproduct $\Delta$ on $\Yhg$. Set
\begin{equation}\label{eq:sfr}
\sfr=\hbar\sum_{\beta\in\Phi_+} x^-_{\beta,0}\otimes x^+_{\beta,0},
\end{equation}
where $x^{\pm}_{\beta,0}\in \g_{\pm\beta}\subset\Yhg$ are root vectors
such that $(x^-_{\beta,0},x^+_{\beta,0})=1$. For any $h\in\h\subset\Yhg$,
define
\begin{equation}\label{eq:sfrh}
\sfr(h) := \ad(h\otimes 1)\cdot \sfr = -\hbar\sum_{\beta\in\Phi_+}
\beta(h)x^-_{\beta,0}\otimes x^+_{\beta,0}\ .
\end{equation}
The  coproduct $\Delta:
\Yhg\to\Yhg\otimes\Yhg$ is then uniquely determined by the following formulae, for  $i\in \bfI$ and $h\in \h$:
\begin{gather*}
\Delta(\xi_{i,0}) = \xi_{i,0}\otimes 1  + 1\otimes \xi_{i,0}, 
\quad 
\Delta(x_{i,0}^{\pm}) 
= x_{i,0}^{\pm}\otimes 1    + 1\otimes x_{i,0}^{\pm},\\[3pt]
\Delta(\Top(h))
= 
\Top(h)\otimes 1   + 1\otimes \Top(h) + \sfr(h).
\end{gather*}
We refer the reader to \cite[\S 4.2]{guay-nakajima-wendlandt} for a
proof that $\Delta$ is an algebra homomorphism. It is immediate
that $\Delta$ is coassociative (see \cite[\S 4.5]{guay-nakajima-wendlandt}).

\subsection{Drinfeld's universal $R$--matrix}\label{ssec:R}

Let $\Delta_s := (\tau_s\otimes 1)\circ \Delta$ and
$\Delta^{\op}_s := (\tau_s\otimes 1)\circ \Delta^{\op}$. Viewing
$\tau_s$ as an algebra homomorphism $\tau_s : \Yhg \to \Yhg[s]$,
Drinfeld \cite[Thm. 3]{drinfeld-qybe} showed that there is a unique $\mathcal{R}(s)\in
\Pseries{\lp\Yhg\otimes\Yhg\rp}{s^{-1}}$ satisfying the following
three conditions:
\begin{enumerate}\itemsep0.25cm
\item $\mathcal{R}(s) = 1\otimes 1 + X(s)$, where
$X(s)\in s^{-1}\Pseries{\Yhg^{\otimes 2}}{s^{-1}}$.
\item For every $a\in \Yhg$, we have
\[
\Delta_s^{\op}(a) = \RR(s) \Delta_s(a) \RR(s)^{-1}.
\]
\item The following cabling identities hold:
\begin{gather*}
\Delta\otimes 1 (\RR(s)) = \RR_{13}(s)\RR_{23}(s)\ ,\\
1\otimes \Delta (\RR(s)) = \RR_{13}(s)\RR_{12}(s)\ .
\end{gather*}
\end{enumerate}
%
\subsection{Intertwining equation}\label{ssec:IE}
As indicated above, a constructive proof of the existence of $\RR(s)$ was recently given in \cite{GTLW} by reassembling it from the components 
in its Gaussian decomposition 
\[
\RR(s) = \RR^+(s)\RR^0(s)\RR^-(s).
\]
The lower triangular part $\RR^-(s)$ is the main object of
study in this paper.
Let us recall its defining properties, following \cite[Thm. 4.1]{GTLW}.
Namely, $\RR^-(s)$ is the unique, zero weight element
of $\Pseries{\lp Y^{\leq}_\hbar(\g)\otimes Y^{\geq}_\hbar(\g)\rp}{s^{-1}}$
satisfying the following two conditions:

\begin{enumerate}\itemsep0.25cm
\item Write $\RR^-(s) = \sum_{\gamma\in\mathsf{Q}_+}\RR^-_\gamma(s)$,
where
\[
\RR^-_\gamma(s)\in \Pseries{\lp Y^{\leq}_\hbar(\g)_{-\gamma}\otimes
Y^{\geq}_\hbar(g)_\gamma\rp}{s^{-1}}.
\]
Then $\RR^-_0(s) = 1\otimes 1$.

\item The following {\em intertwining equation} holds, for every $h\in\h$:
\begin{equation}\label{eq:IE}
[\Top(h)\otimes 1 + 1\otimes \Top(h) + s h\otimes 1,
\RR^-(s)] = \RR^-(s)\sfr(h)\ .
\end{equation}
\end{enumerate}
Here we note that \eqref{eq:IE} is equivalent to the relation (4.1) in \cite{GTLW}, which is written in terms of the deformed Drinfeld coproduct on $\Yhg$ and $\Delta_s$ from Section \ref{ssec:R} above. 
Setting $\mathcal{D}(h;s) = 
\ad(\Top(h)\otimes 1 + 1\otimes \Top(h) + s h\otimes 1)$, it may also be written in terms of the elements $\RR^-_\gamma(s)$ as
\begin{equation}\label{eq:IE-block}
\mathcal{D}(h;s)\cdot \RR^-_\gamma(s) = -\hbar\sum_{
\begin{subarray}{c} \alpha\in \Phi_+\  \\
\gamma-\alpha\in \mathsf{Q}_+
\end{subarray}}
\RR^-_{\gamma-\alpha}(s) x^-_{\alpha,0}\otimes x^+_{\alpha,0}\ .
\end{equation}
Together with the initial condition $\RR^-_0(s)=1\otimes 1$, 
this equation defines $\RR^-_\beta(s)$ inductively on the height of
$\beta$; see \cite[Eqn.~(4.6)]{GTLW}. The sum
$\sum_{\beta\in\mathsf{Q}_+} \RR^-_\beta(s)$ is a well--defined
element of $\Pseries{\Yhg^{\otimes 2}}{s^{-1}}$
which solves \eqref{eq:IE} and, by Part (2) of \cite[Thm. 4.1]{GTLW}, lies in
$\Pseries{\lp Y^-_\hbar(\g)\otimes Y^+_\hbar(\g)\rp}{s^{-1}}$.

\section{Non-existence of $\RR^-(s)$ factorizations}\label{sec:impossible}

The aim of this section is to show that the unique solution
$\RR^-(s)$ of \eqref{eq:IE} does not admit any factorization
over the set of positive roots $\Phi_+$. 
From now onwards, we assume that $\mathrm{rank}(\g)>1$.

\subsection{Notation}\label{ssec:not}
For a given total order $<$ on $\Phi_+$ and a collection of
elements $\{A^{(\alpha)}\}_{\alpha\in\Phi_+}$ lying in some associative
algebra, we set
\[
\prod_{\alpha\in\Phi_+}^< A^{(\alpha)} := A^{(\beta_1)}\cdots A^{(\beta_N)}\ ,
\]
where $\Phi_+ = \{\beta_1<\ldots < \beta_N\}$.
In addition, for each $\gamma\in\mathsf{Q}_+$ we let $\mathcal{P}(\gamma)$ denote
the set of partitions of $\gamma$ as a sum of positive roots:
\begin{equation*}
\mathcal{P}(\gamma) = \left\{ \ul{k}=(k_{\alpha})_{\alpha\in\Phi_+}
\in\Z_{\geq 0}^{\Phi_+} : \gamma = \sum_{\alpha\in\Phi_+}
k_\alpha \alpha\right\}.
\end{equation*}
We further record the following symbolic identity for reference later, where $\{X^{\alpha}_n : \alpha\in\Phi_+, n\in\N\}$ is an arbitrary
collection of non--commuting variables:
\begin{equation}\label{eq:symb}
\prod_{\alpha\in\Phi_+}^< \lp \sum_{n=0}^\infty X^{(\alpha)}_n \rp
= \sum_{\gamma\in\mathsf{Q}_+} \lp
\sum_{\ul{k}\in\mathcal{P}(\gamma)} \prod_{\alpha\in\Phi_+}^<
X^{(\alpha)}_{k_\alpha} \rp.
\end{equation}
\subsection{Total order}\label{ssec:order}
Fix a total order $\prec$ on $\Phi_+$ satisfying the
following condition: there exist two simple roots $\alpha_i,\alpha_j$
such that
\begin{itemize}\itemsep0.25cm
\item $\alpha_i+\alpha_j\in \Phi_+$ and $\alpha_i+\ell\alpha_j\not\in\Phi_+$
for every $\ell\in\Z_{\geq 2}$.
\item $\alpha_i \prec \alpha_i+\alpha_j \prec \alpha_j$.
\end{itemize}

\begin{rem}
Recall that a total order $<$ on $\Phi_+$ is said to be {\em convex}
(or {\em normal})
if for every $\alpha,\beta\in\Phi_+$ such that $\gamma=\alpha + \beta\in\Phi_+$,
either $\alpha<\gamma<\beta$, or $\beta<\gamma<\alpha$.
Thus, the order $\prec$ considered above could be any
convex ordering, if the root system is not $\mathsf{B}_2$
or $\mathsf{G}_2$. In the $\mathsf{B}_2,\mathsf{G}_2$ cases,
we can take, for instance, the following convex order,
where $\alpha_1$ is the short simple root:
\begin{itemize}
\item[($\mathsf{B}_2$)]
$\alpha_1\prec 2\alpha_1+\alpha_2\prec \alpha_1+\alpha_2\prec \alpha_2$.
\item[($\mathsf{G}_2$)]
$\alpha_1\prec 3\alpha_1+\alpha_2 \prec 2\alpha_1+\alpha_2
\prec 3\alpha_1+2\alpha_2\prec \alpha_1+\alpha_2\prec \alpha_2$.
\end{itemize}
\end{rem}

\subsection{Block elements}\label{ssec:block}

Assume that we are given an arbitrary collection 
$\{F^{(\alpha)}_n(s) : \alpha\in\Phi_+, n\in\N\}$, where
\[
F^{(\alpha)}_n(s) \in \Pseries{\lp Y^{\leq}_\hbar(\g)_{-n\alpha}
\otimes Y^{\geq}_\hbar(\g)_{n\alpha}\rp}{s^{-1}}
\]
and $F^{(\alpha)}_0(s) = 1\otimes 1$. In accordance with
\eqref{eq:symb}, define:
\begin{equation}\label{eq:partitions}
F_\gamma(s) = \sum_{\ul{k}\in\mathcal{P}(\gamma)}
\prod_{\alpha\in\Phi^+}^\prec
F^{(\alpha)}_{k_\alpha}(s)\ , \qquad
F_\gamma(s)\in\Pseries{\lp Y^{\leq}_\hbar(\g)_{-\gamma}
\otimes Y^{\geq}_\hbar(\g)_\gamma\rp}{s^{-1}}.
\end{equation}

\subsection{Main theorem}\label{ssec:impossible}

To state our main theorem, let us introduce some auxiliary terminology. Let $\mathcal{J}(s)$ be an arbitrary weight zero element of  $\Pseries{(Y^{\leq}_\hbar(\g)
\otimes Y^{\geq}_\hbar(\g))}{s^{-1}}$, and write $\mathcal{J}(s)=\sum_{\gamma\in\mathsf{Q}_+}\mathcal{J}_\gamma(s)
$ with 
\begin{equation*}
\mathcal{J}_\gamma(s)\in (Y^{\leq}_\hbar(\g)_{-\gamma}\otimes
Y^{\geq}_\hbar(\g)_\gamma)[\![s^{-1}]
\!].
\end{equation*}
 Then, for a fixed $\alpha\in \Phi_+$, we say that $\mathcal{J}(s)$ has $\Z\alpha$-\textit{support} if $\mathcal{J}_\gamma(s)=0$ for $\gamma\notin \Z\alpha$ and, in addition, $\mathcal{J}_0(s)=1\otimes 1$.

\begin{thm}\label{thm:impossible}
For any total ordering $\prec$ as in \S \ref{ssec:order}, and 
$\{F^{(\alpha)}_n(s)\}_{\alpha\in\Phi_+,n\in\N}$
as in \S \ref{ssec:block}, the elements 
$\{F_\gamma(s)\}_{\gamma\in\mathsf{Q}_+}$
do not satisfy
the intertwining equation \eqref{eq:IE-block}. Consequently, $\RR^-(s)$ does not admit a factorization of the form 
\begin{equation*}
\RR^-(s)=\prod_{\alpha\in\Phi_+}^\prec  \mathcal{J}^{(\alpha)}(s),
\end{equation*}
where, for each $\alpha\in \Phi_+$, $\mathcal{J}^{(\alpha)}(s)$ is a weight zero element of $\Pseries{(Y^{\leq}_\hbar(\g)
\otimes Y^{\geq}_\hbar(\g))}{s^{-1}}$ with $\Z\alpha$-support.
\end{thm}

\begin{proof}[{\sc Proof}]
The proof is by contradiction, which we split into three
elementary steps whose details are worked out in 
\S \ref{ssec:pf-simple}--\ref{ssec:pf-rk2} below.
The struture of our argument is as follows.
Assume that \eqref{eq:IE-block} holds for $\{F_\gamma(s)\}$.
\begin{enumerate}\setlength{\itemsep}{5pt}
\item For each simple root $\alpha_k$, 
$\{F^{(\alpha_k)}_n(s)\}_{n\in\N}$ can be explicitly
computed, as in rank $1$ case. We only need the first
two terms of $F^{(\alpha_k)}_1(s)$, which are obtained
in \S \ref{ssec:pf-simple}.

\item In \S \ref{ssec:pf-comm} we show that
$F^{(\alpha_j)}_n(s)$ commutes with $x^-_{ij,0}\otimes
x^+_{ij,0}$. Here $\alpha_i,\alpha_j$ are the simple
roots satisfying conditions imposed on $\prec$ in
\S \ref{ssec:order} above, and $x_{ij,0}^\pm$
are the root vectors corresponding to $\alpha_i+\alpha_j$.

\item A simple rank $2$ computation is then carried out
in \S \ref{ssec:pf-rk2} to show that $x^-_{ij,0}\otimes
x^+_{ij,0}$ does not commute with $F^{(\alpha_j)}_1(s)$,
thus obtaining the desired contradiction. \qedhere
\end{enumerate}
\end{proof}

\begin{rem}
In the setup of \S \ref{ssec:block}, we
did not assume that $\sum_{\gamma\in\mathsf{Q}_+} F_\gamma(s)$
exists as an element of $\Pseries{\Yhg^{\otimes 2}}{s^{-1}}$.
This, in fact, is a consequence of \eqref{eq:IE-block},
as observed in \cite[Eqns. (4.6)--(4.7)]{GTLW}. Namely, if
$\{A_\gamma(s)\}_{\gamma\in\mathsf{Q}_+}$ solve
\eqref{eq:IE-block} and $A_0(s)=1\otimes 1$, 
then $A_\gamma(s)$ is divisible by $s^{-\nu(\gamma)}$, where
\begin{equation*}
\nu(\gamma) = \mathrm{min}\left\{ k\in\Z_{\geq 0}\ |\ 
\gamma = \beta_1+\cdots+\beta_k,\text{ where } \beta_1,\ldots,\beta_k\in\Phi_+ \right\}.
\end{equation*}
For an interpretation of this fact in terms of dual bases and the Yangian double, we refer the reader to Corollary 9.9 of \cite{Curtis-DYg}.
\end{rem}

\subsection{Simple roots}\label{ssec:pf-simple}
Let $\alpha_k\in\Phi_+$ be a simple root. Then the defining equation \eqref{eq:partitions}
for $F_\gamma(s)$ implies that $F_{n\alpha_k}(s) = F^{(\alpha_k)}_n(s)$, for every
$n\in\N$. The
intertwining equation \eqref{eq:IE-block} with $\gamma=n\alpha_k$
then becomes:
\begin{equation}\label{eq:nk}
\mathcal{D}(h;s)\cdot F^{(\alpha_k)}_n(s) = 
-\hbar\alpha_k(h) F^{(\alpha_k)}_{n-1}(s) x^-_{k,0}\otimes x^+_{k,0}.
\end{equation}
\begin{rem}
We would like to point out that this equation is precisely
the one defining $\RR^-(s)$ for $\sl_2$ whose explicit formula
is given in \cite[Thm. 5.5]{GTLW}. For the purposes of our
proof, it is enough to know the coefficient of $s^{-2}$ in 
$F^{(\alpha_k)}_1(s)$. We include this easy computation below, 
for completeness.
\end{rem}
Using $F^{(\alpha_k)}_0(s)=1\otimes 1$, the $n=1$ case
of equation \eqref{eq:nk} is the following:
\[
[\Top(h)\otimes 1 + 1\otimes \Top(h) + s h\otimes 1,
F^{(\alpha_k)}_1(s)] = -\hbar\alpha_k(h) x^-_{k,0}\otimes x^+_{k,0}.
\]

Comparing coefficients of $s^0$ and $s^{-1}$, and using
the commutation relation \eqref{eq:T}, we obtain
\begin{equation}\label{eq:1-jet}
F^{(\alpha_k)}_1(s) = \hbar s^{-1} \lp x^-_{k,0}\otimes x^+_{k,0}
 + (-x^-_{k,1}\otimes x^+_{k,0} + x^-_{k,0}\otimes x^+_{k,1})s^{-1}
+ \ldots \rp.
\end{equation}
In fact, though not needed in the present article, it follows easily from \eqref{eq:T} and the above relation for $F^{(\alpha_k)}_1(s)$ that 
\begin{equation*}
F^{(\alpha_k)}_1(s)=\sum_{n\geq 0} x_{k,n}^-\otimes \partial_s^{(n)} x_k^+(s),
\end{equation*}
where $x_k^+(s)=\hbar\sum_{r\geq 0} x_{k,r}^+ s^{-r-1}$ and $\partial_s^{(n)}=\frac{1}{n!}\partial_s^{n}$, with $\partial_s$ the partial derivative operator with respect to $s$. This formula may be found in the proof \cite[Prop.~7.1]{Curtis-DYg}, above Remark 7.2 therein. 

\subsection{Commutativity relation}\label{ssec:pf-comm}
Now, let $\alpha_i,\alpha_j\in\Phi_+$ be the simple roots for which
the condition imposed in \S \ref{ssec:order} holds. For notational
convenience, we will write $\alpha_{ij} = \alpha_i+\alpha_j\in\Phi_+$.
Similarly, $x_{ij,0}^{\pm}$ will denote the root vectors corresponding
to the positive root $\alpha_{ij}$. We will also abbreviate
\[
\sfr_a(h) = -\hbar\alpha_a(h)x^-_{a,0}\otimes x^+_{a,0}, \text{ for }a=i,j
\text{ or } a=ij.
\]

Equations \eqref{eq:IE-block}
and \eqref{eq:partitions} for $\gamma=\alpha_i+\ell\alpha_j\in\mathsf{Q}_+$
take the following form:
\begin{gather}\label{eq:i+lj-1}
\mathcal{D}(h;s)\cdot F_{\alpha_i+\ell\alpha_j}(s)
 = F_{\ell\alpha_j}(s)\sfr_i(h) + F_{(\ell-1)\alpha_j} \sfr_{ij}(h)
+ F_{\alpha_i+(\ell-1)\alpha_j}(s) \sfr_j(h),\\[3pt]
\label{eq:i+lj-2}
F_{\alpha_i+\ell\alpha_j}(s) = F^{(\alpha_{ij})}_1(s)F^{(\alpha_j)}_{\ell-1}(s)
+ F^{(\alpha_i)}_1(s)F^{(\alpha_j)}_\ell(s).
\end{gather}

Combining these two equations,
using the fact that $\mathcal{D}(h;s)$ is a derivation, and
\eqref{eq:nk} for $n=1$ above, we get:
\begin{equation*}
\lp\mathcal{D}(h;s)\cdot F^{(\alpha_{ij})}_1(s)\rp
F^{(\alpha_j)}_{\ell-1}(s) = 
F^{(\alpha_j)}_{\ell-1}(s)\sfr_{ij}(h) + 
\left[
F^{(\alpha_j)}_{\ell}(s),\sfr_i(h)
\right].
\end{equation*}

Now choose $h\in\alpha_i^\perp \subset \h$ so that $\sfr_i(h)=0$,
and take $\ell=1$ to get $\mathcal{D}(h;s)\cdot F^{(\alpha_{ij})}_1(s)
 = \sfr_{ij}(h)$. Substitute this back into the equation above
to obtain:
\[
\left[
\sfr_{ij}(h), F^{(\alpha_j)}_{\ell-1}(s)
\right] = 0,\ \forall\  h\in\alpha_i^\perp,\ \ell\in\Z_{\geq 1}.
\]

Take $h\in\alpha_i^\perp$ such that $\alpha_j(h)\neq 0$ and therefore,
$\sfr_{ij}(h)$ is a non--zero scalar multiple of $x_{ij,0}^-\otimes
x_{ij,0}^+$. We get:
\[
\left[x_{ij,0}^-\otimes x_{ij,0}^+, F^{(\alpha_j)}_n(s)
\right] = 0,\ \forall\ n\in\N.
\] 

Take $n=1$ and the coefficient of $s^{-2}$ using \eqref{eq:1-jet}
to get:
\begin{equation}\label{eq:contra}
\left[
x_{ij,0}^-\otimes x_{ij,0}^+\ ,\ x^-_{j,0}\otimes x^+_{j,1}
- x^-_{j,1}\otimes x^+_{j,0}
\right] = 0.
\end{equation}

\subsection{Rank $2$ computation}\label{ssec:pf-rk2}
 Let us now restrict
our attention to the rank $2$ subsystem generated by $\alpha_i$
and $\alpha_j$. For notational simplicity, we will replace
$i,j$ by $1,2$ and our $2\times 2$ Cartan matrix is of
the following form:
\[
\bfA = \left[\begin{array}{rr}
2 & -p \\ -1 & 2
\end{array}\right],\ \ p=1,2\ \text{or}\ 3.
\]
In this case, $d_1=1$ and $d_2=p$. Let us write $\alpha_3 = \alpha_1+\alpha_2$
and take the following root vectors $x^{\pm}_{3,0}\in\g_{\pm\alpha_3}$,
so that $(x_{3,0}^+,x_{3,0}^-) = 1$:
\[
x^-_{3,0} = [x_{2,0}^-,x_{1,0}^-] \aand
x^+_{3,0} = \frac{1}{p} [x_{1,0}^+,x_{2,0}^+].
\]
Note that by the Serre relations, $x_{3,0}^{\pm}$ commutes with
$x^{\pm}_{2,0}$.

\noindent {\bf Claim.} The defining relations of $\Yhg$
imply that
\begin{equation}\label{eq:contra2}
[x^-_{3,0}\otimes x^+_{3,0}, x_{2,0}^-\otimes x^+_{2,1}
- x^-_{2,1}\otimes x^+_{2,0}] = -2p\hbar
x^-_{3,0}x^-_{2,0}\otimes x^+_{3,0}x^+_{2,0}\,,
\end{equation}
which contradicts \eqref{eq:contra}.

\begin{proof}[Proof of the claim]
Using \ref{Y4} and \ref{Y6}, we obtain the two identities
\begin{gather*}
[x^\pm_{2,1},x^\pm_{1,0}] = [x^\pm_{2,0},x^\pm_{1,1}]
\mp \frac{p\hbar}{2} (x^\pm_{2,0}x^\pm_{1,0}
+ x^\pm_{1,0}x^\pm_{2,0}),\\
[x^\pm_{2,1},[x^\pm_{2,0},x^\pm_{1,0}]]
= -[x^\pm_{2,0},[x^\pm_{2,1},x^\pm_{1,0}]].
\end{gather*}
Combining these two equations,\ and using the fact that, 
by the Serre relations, $x^{\pm}_{2,0}$ commutes with
$[x^{\pm}_{2,0},x^{\pm}_{1,k}]$ for all $k\geq 0$, we obtain 
\begin{equation*}
[x^{\pm}_{2,1},[x^{\pm}_{2,0},x^{\pm}_{1,0}]]
= \pm \hbar p x^{\pm}_{2,0}[x^{\pm}_{2,0},x^{\pm}_{1,0}].
\end{equation*}
Now we can carry out the following computations:
\begin{align*}
[x^+_{3,0},x^+_{2,1}] &= \frac{1}{p}[[x^+_{1,0},x^+_{2,0}],x^+_{2,1}]
= \frac{1}{p} [x^+_{2,1},[x^+_{2,0},x^+_{1,0}]] 
= \hbar x^+_{2,0}[x^+_{2,0},x^+_{1,0}]\\
&= -p\hbar x^+_{2,0}x^+_{3,0},\\[5pt]
[x^-_{3,0},x^-_{2,1}] &= [[x^-_{2,0},x^-_{1,0}],x^-_{2,1}]
= -[x^-_{2,1},[x^-_{2,0},x^-_{1,0}]] 
= p\hbar x^-_{2,0}[x^-_{2,0},x^-_{1,0}] \\
&= p\hbar x^+_{2,0}x^+_{3,0}.
\end{align*}
Hence, the left--hand side of \eqref{eq:contra2} simplifies to
\[
x^-_{3,0}x^-_{2,0}\otimes [x^+_{3,0},x^+_{2,1}]
- [x^-_{3,0},x^-_{2,1}]\otimes x^+_{3,0}x^+_{2,0}
= -2p\hbar x^-_{3,0}x^-_{2,0}\otimes x^+_{3,0}x^+_{2,0}
\]
as claimed. \qedhere
\end{proof}

\section{Conclusions}\label{sec:rks}

In this last section we discuss the Khoroshkin--Tolstoy
conjecture that motivated Theorem~\ref{thm:impossible}.
Moreover, we briefly review the well--known factorization formulae
for the universal $R$--matrices of Drinfeld--Jimbo quantum groups
associated to finite or affine Lie algebras. We then 
observe that our result does not exclude the existence of {\em this
kind} of factorization formulae for the Yangian's $R$--matrix, which
therefore remains an open and challenging problem.

\subsection{The Khoroshkin--Tolstoy conjecture}\label{ssec:KTC}
Let us now explain how Theorem~\ref{thm:impossible} relates to 
the conjectural formula \cite[Eqn. (5.43)]{khoroshkin-tolstoy}. In this section alone, we 
assume that $\hbar$ is a formal variable, following the conventions of \cite{Curtis-DYg}; see in particular \S1.4 therein. We shall return to the case where $\hbar\in \C^\times$ at the end of the section. 

Let $\mathrm{D}\Yhg$ denote the Yangian double, as defined \cite[Defn.~2.1]{khoroshkin-tolstoy} and 
\cite[Defn.~4.1]{Curtis-DYg}. That is, $\mathrm{D}\Yhg$ is the unital, associative
$\C[\![\hbar]\!]$-algebra topologically generated by $\{\xi_{i,r},x_{i,r}^\pm\}_{i\in\bfI,r\in\Z}$,
subject to the same family of relations given in \S \ref{ssec: yangian}
above, where the subscripts $r,s$ now range over $\Z$.
It was proven in Theorem 8.4 of \cite{Curtis-DYg} that $\mathrm{D}\Yhg$ provides a realization of the (restricted) quantum double of the Yangian $\Yhg$, as conjectured in \cite[\S2]{khoroshkin-tolstoy}.   Let $\boldR$
be the universal $R$--matrix of $\mathrm{D}\Yhg$, and let $\boldR^-$ be the
lower triangular factor in its Gaussian decomposition, as first considered in \cite[\S5]{khoroshkin-tolstoy}. We refer the reader to \cite[\S9.3]{Curtis-DYg} for the precise definitions of these elements.

Following the conventions of \cite[\S 5]{khoroshkin-tolstoy},
let $\Sigma_- = \{-\gamma+k\delta : \gamma\in\Phi_+, k\geq 0\}$, where
$\delta$ is the imaginary root of the (affine) root system of $\wh{\g}$.
Choose an arbitrary total order $<$ on $\Sigma_-$ satisfying the
following two conditions:
\begin{itemize}\itemsep0.25cm
\item if $\alpha,\beta,\gamma=\alpha+\beta\in\Sigma_-$, then
either $\alpha<\gamma<\beta$ or $\beta<\gamma<\alpha$\,;
\item $-\gamma+\ell\delta < -\gamma+k\delta$, for $k<\ell$.
\end{itemize}
Given such a total order $<$, it was conjectured in \cite{khoroshkin-tolstoy} (see \cite[Eqn. (5.43)]{khoroshkin-tolstoy}) that $\boldR^-$ admits a multiplicative factorization 
\begin{equation}\label{eq:KTC}
\boldR^- = \prod_{\beta\in\Sigma_-}^<
\exp\lp -\hbar \Omega_\beta\rp,
\end{equation}
where $\Omega_\beta$ is an explicitly defined simple tensor $\Omega_\beta=\omega_\beta^-\otimes \omega_\beta^+$ with $\omega_\beta^-\in Y_\hbar^-(\g)\subset \mathrm{D}\Yhg$, $\omega_\beta^+$ an element of the dual Yangian $Y_\hbar^-(\g)^\star\subset \mathrm{D}\Yhg$ (see \cite[\S6.5]{Curtis-DYg}) and, if $\beta=-\gamma+k\delta$, the factors $\omega_\beta^-$ and $\omega_\beta^+$  have degree $(k,-\gamma)$ and   $(-k-1,\gamma)$ with respect to the standard $\Z\times \mathsf{Q}$ grading on $\mathrm{D}\Yhg$, respectively. This conjecture extends Lemma 5.1 of \cite{khoroshkin-tolstoy}, which established \eqref{eq:KTC} for $\g=\sl_2$. In this case one has $\Omega_{-\alpha_i+k\delta}=x_{i,k}^-\otimes x_{i,-k-1}^+$, where $\bfI=\{i\}$.

In order to relate this conjectural expression to the
results of Section \ref{sec:impossible} above, we make the following
choice for the total order on $\Sigma_-$, where we once again assume $\g\ncong \sl_2$. If $\g$ is not of type $\mathsf{B}_2$ or $\mathsf{G}_2$, we fix $<_1$ to be an arbitrary convex order on the
set $\Phi_+$. In the $\mathsf{B}_2$ and $\mathsf{G}_2$ cases,  we take $<_1$ to be the order given
in \S \ref{ssec:order}. We then extend $<_1$ to $\Sigma_-$ as follows:
\begin{equation*}
-\alpha+\ell\delta <_1 -\gamma+k\delta, \text{ if either }
\alpha<_1\gamma \text{; or } \alpha=\gamma \text{ and } k<\ell.
\end{equation*}
We remark that this is the same ordering
utilized in \cite[Thm. 6.2]{soloviev} for $\g=\sl_3$.
With this choice of ordering, equation \eqref{eq:KTC}
becomes:
\begin{equation}\label{eq:ST}
\boldR^- = \prod_{\alpha\in\Phi_+}^{<_1}
\lp
\prod_{k\geq 0}^{\leftarrow} \exp\left(-\hbar \Omega_{-\alpha+k\delta}\right)
\rp = \prod_{\alpha\in\Phi_+}^{<_1} \mathcal{J}^{(\alpha)},
\end{equation}
where $\mathcal{J}^{(\alpha)}$ is defined to be product over $k\geq 0$ that appears within the parentheses on the right-hand side of the first equality. To pass from this expression back to the setting of Section \ref{sec:impossible}, we recall that the shift homomorphism $\tau_s$ from Section \ref{ssec:R} extends to a ($\Z\times \mathsf{Q}\,$-graded) algebra homomorphism $\Phi_s$ from $\mathrm{D}Y_\hbar(\g)$ into an algebra contained in the space $\Yhg[\![s^{\pm 1}]\!]$; see \cite[Thm.~4.3]{Curtis-Phi} and \cite[Thm.~4.6]{Curtis-DYg}. In particular, $\Phi_s$ satisfies $\Phi_s(\omega_{-\alpha+k\delta}^+)\in s^{-k-1}Y_\hbar^+(\g)_\alpha[\![s^{-1}]\!]$ for all $\alpha\in \Phi_+$ and $k\geq 0$. 

In Theorem 9.6 of \cite{Curtis-DYg}, it is shown that $(\id\otimes \Phi_{-s})(\boldR^-)$ coincides with $\RR^-(s)$; see also \cite{Stukopin07} together with \cite[\S1.3]{Curtis-DYg}. Thus, \eqref{eq:ST}
implies that $\RR^-(s)$ can be written as the ordered product 
\begin{equation*}
\RR^-(s)=\prod_{\alpha\in\Phi_+}^{<_1} \mathcal{J}^{(\alpha)}(s),
\end{equation*}
where, for each $\alpha\in \Phi_+$, $\mathcal{J}^{(\alpha)}(s):=(\id\otimes \Phi_{-s})(\mathcal{J}^{(\alpha)})$. Note that, in the terminology of Section \ref{ssec:impossible}, $\mathcal{J}^{(\alpha)}(s)$ is necessarily a weight zero element of $\Pseries{(Y^{-}_\hbar(\g)
\otimes Y^{+}_\hbar(\g))}{s^{-1}}$ with $\Z\alpha$-support. Here we emphasize that, by Proposition A.1 of \cite{Curtis-DYg}, the above factorization must hold regardless of whether $\hbar$ is viewed as a formal variable or an arbitrary non-zero complex number, as in Sections \ref{sec:YR} and \ref{sec:impossible} above. However, this is impossible as shown
in Theorem~\ref{thm:impossible}, proving that \eqref{eq:KTC}
is false.

\subsection{$R$--matrices of quantum affine algebras}\label{ssec:qla}
For Drinfeld--Jimbo quantum groups, the $R$--matrix can
be expressed as an ordered product of $R$--matrices of $U_q(\sl_2)$
associated to positive roots \cite{kirillov-reshetikhin,
LS1990, LS1991}. The root subalgebras are constructed relying on
Lusztig's braid group action \cite{lusztig-book} and
the factorization is a consequence of the relation
between the coproduct and the quantum Weyl group operators
\cite{kirillov-reshetikhin, LS1991, lusztig-book}.

For quantum affine algebras, this product
ranges over the set of positive affine roots $\wh{\Phi}_+$,
see \eg \cite{levendorskii-soibelman-stukopin-93,beck-braid,beck-PBW} and
\cite[Thm.~2]{damiani}.
Let $\mathbf{R}$ the
universal $R$--matrix of $U_q(\wh{\g})$. Then,
up to a Cartan correction, which for simplicity is suppressed in the current
discussion, one has
\begin{equation}\label{eq:DJ-Rmx}
\mathbf{R} = \prod_{\alpha\in\wh{\Phi}_+}
^{\prec'} \exp_{q_\alpha}\lp (q_\alpha-q_{\alpha}^{-1})
E_\alpha\otimes F_\alpha\rp\,.
\end{equation}
For full details about the total order $\prec'$ on $\wh{\Phi}_+$,
we refer the reader to \cite{beck-PBW,damiani}. Here, we only
need few of its salient features. 
Recall that the set of affine positive roots is given by $\wh{\Phi}_+
= \wh{\Phi}_+^{\re} \sqcup \wh{\Phi}_+^{\im}$, 
where
$\wh{\Phi}_+^{\im} = \Z_{\geq 1}\delta$ and $\wh{\Phi}_+^{\re} = \wh{\Phi}_{+,+}^{\re}\cup \wh{\Phi}_{+,-}^{\re}$ with
\begin{equation*}
\wh{\Phi}_{+,+}^{\re} = \{\alpha+k\delta : \alpha\in\Phi_+,k\in\N\}
\quad\text{ and }\quad
\wh{\Phi}_{+,-}^{\re} = \{-\alpha+\ell\delta : \alpha\in\Phi_+, \ell\in\Z_{\geq 1}\}.
\end{equation*}
Then, $\prec'$ has the following properties.
\begin{enumerate}\itemsep0.25cm
\item $x\prec'y\prec'z$ for $x\in \wh{\Phi}^{\re}_{+,+}$,
$y\in \wh{\Phi}^{\im}_+$ and $z\in \wh{\Phi}^{\re}_{+,-}$.

\item The total order $\prec'$ restricted to $\wh{\Phi}^{\re}_{+,-}$
is convex.

\item For any $z_1, z_2\in\wh{\Phi}^{\re}_{+,-}$, the interval
$\{z : z_1\prec' z\prec' z_2\}$ is finite.
\end{enumerate}
Note that the property (3) does not hold for the (general) 
total orders considered in \cite{khoroshkin-tolstoy}, the ones 
featuring in Theorem~\ref{thm:impossible} and ultimately in the equation 
\eqref{eq:ST}.

By property (1) and equation \eqref{eq:DJ-Rmx}, the universal $R$--matrix
$\mathbf{R}$ factors into three components,
which we denote by $\mathbf{R}^+$, $\mathbf{R}^0$ and $\mathbf{R}^-$
respectively.
By the results of \cite{sachin-valerio-III}, we expect that an analogue of Theorem~\ref{thm:impossible} holds for $\mathbf{R}^-$, \ie 
$\mathbf{R}^-$ cannot be expressed as an ordered product of the
form \eqref{eq:ST}.

Conversely, our results do not exclude the possibility of a
factorization of $\boldR^-$ or $\RR^-(s)$ similar to that of $\mathbf{R}^-$,
\ie a factorization of the form \eqref{eq:DJ-Rmx} for an order satisfying the properties above. 
To the best of our knowledge, no such result in the case of Yangians 
or Yangian doubles is known, and remains an interesting
and challenging open problem.


\bibliographystyle{amsplain}
\providecommand{\bysame}{\leavevmode\hbox to3em{\hrulefill}\thinspace}
\providecommand{\MR}{\relax\ifhmode\unskip\space\fi MR }
\providecommand{\MRhref}[2]{%
  \href{http://www.ams.org/mathscinet-getitem?mr=#1}{#2}
}
\providecommand{\href}[2]{#2}

\end{document}